\documentclass[11pt,reqno]{amsart}
\usepackage{amsmath, amssymb, amsthm}
\usepackage{url}
\usepackage[breaklinks]{hyperref}
\renewcommand{\(}{\left\(}
\renewcommand{\)}{\right\)}
\renewcommand{\[}{\left\[}
\renewcommand{\]}{\right\]}

\numberwithin{equation}{section}
 \theoremstyle{plain}
\newtheorem{theorem}{Theorem}[section]
\newtheorem{lemma}[theorem]{Lemma}
\newtheorem{remark}[]{Remark}

\newtheorem{corollary}[theorem]{Corollary}

   \makeatletter
\def\proof{\@ifnextchar[{\@oproof}{\@nproof}}
\def\@oproof[#1][#2]{\trivlist\item[\hskip\labelsep\textit{#2 Proof of\
#1.}~]\ignorespaces}
\def\@nproof{\trivlist\item[\hskip\labelsep\textit{Proof.}~]\ignorespaces}
\makeatother

\begin{document}
\title[A Combinatorial Proof of Partition Identities for the Number of Distinct Part Sizes ]{A Combinatorial Proof of Partition Identities for the Number of Distinct Part Sizes}

\thanks{$2020$ \textit{Mathematics Subject Classification.} Primary 11P84, 05A17; Secondary 11P81.}

\author{Archit Agarwal}
\address{Archit Agarwal, School of Mathematical Sciences, National Institute of Science Education Research, Bhubaneswar, Odisha, 752050, India.}
\email{archit.agrw@gmail.com, archit.agrw@niser.ac.in}

\author{Meghali Garg}
\address{Meghali Garg, Department of Mathematics, Indian Institute of Science, Bangalore, Karnataka, 560012, India.}
\email{meghaligarg.2216@gmail.com,
meghali1@iisc.ac.in}
\begin{abstract}
In 1984, Bressoud and Subbarao established an interesting weighted partition identity connecting the partition function and the divisor function. Recently, the first author, jointly with Bhoria, Eyyunni, and Maji, extended the work of Bressoud and Subbarao and obtained the following analogous weighted partition identity. For any $n \in \mathbb{N}$, they obtained
\begin{equation*}
p^{(2)}(n) = \sum_{\pi \in \mathcal{D}(n)} (-1)^{ \# (\pi) }s(\pi) \left( \ell(\pi) - s(\pi) \right).
\end{equation*}
In the present paper, we provide a  combinatorial proof of the above identity and extend the result for $p^{(k)}(n)$, where $p^{(k)}(n)$ denotes the number of partitions of $n$ into exactly $k$ distinct part sizes.
\end{abstract}
\maketitle
\section{Introduction}
Partition identities involving divisor functions have a long and fruitful history in the theory of partitions and $q$-series. A particularly interesting class of identities connects weighted sums over partitions into distinct parts with classical arithmetic functions such as the divisor function. Such identities often admit both analytic and combinatorial interpretations, and combinatorial proofs provide additional insight by revealing the cancellation mechanisms underlying the corresponding generating-function identities. In 1984, Bressoud and Subbarao \cite{BS84} found one such identity by comparing the coefficient of $q^n$ from an identity of Uchimura \cite[Theorem 2]{uchimura81}. Before discussing their identity, we introduce some notation that will be used throughout the paper.
\begin{itemize}
\item $\pi$: an integer partition,

\item $p(n)$:= the number of integer partitions of $n$,

\item $p^{(t)}(n)$:= the number of partitions of $n$ into exactly $t$ distinct part sizes,

\item $ s(\pi):=$ the smallest part of $\pi$,

\item $\ell(\pi):=$ the largest part of $\pi$,

\item $\#(\pi):=$ the number of parts of $\pi$,

\item $|\pi|:=$ size of partition $\pi$,
  
\item $\nu(\pi):=$ the number of parts of $\pi$ without multiplicity,

%\item $\nu(j):=$ the number of times the integer $j$ occurs in a partition,

%\item $\overline{p}(n):=$ the number of overpartitions of $n$,

\item $\mathcal{P}(n):=$ collection of all integer partitions of $n$,

\item $\mathcal{D}(n):=$ collection of all partitions of $n$ into distinct parts,

%\item $\mathcal{P}_{o}(n):=$ collection of all overpartitions of $n$,

%\item $\mathcal{P}^{*}(n):=$ partitions into consecutive integers with smallest part $1$.
\end{itemize}
%\begin{align}\label{uchimura}
%\sum_{n=1}^{\infty} n q^n (q^{n+1})_{\infty} = \sum_{n=1}^{\infty}\frac{(-1)^{n-1}q^{\frac{n(n+1)}{2}}}{(1-q^n)(q)_n} = \sum_{n=1}^{\infty}\frac{q^n}{1-q^n},
%\end{align}
%where $\displaystyle (a)_n:=(a;q)_n= \prod_{k=0}^{n-1}(1-aq^k)$
%A few years later, Bressoud and Subbarao \cite{5} provided a beautiful partition-theoretic interpretation of this identity. We describe their result next. Before doing so, we introduce the notation that will be used throughout the paper.
Let $d(n)$ count the number of positive divisors of $n$. Bressoud and Subbarao \cite{BS84} proved  for any natural number $n$,
\begin{equation}\label{BS}
\sum_{ \pi \in \mathcal{D}(n)  } (-1)^{ \# (\pi)-1 } s(\pi)=d(n).  
\end{equation}
They proved this identity using a combinatorial argument. Later in 1995, Fokkink, Fokkink and Wang \cite{ffw95} also rediscovered this identity. Bressoud and Subbarao not only proved \eqref{BS} but also gave a more general result involving generalized divisor function. Mainly, they proved for any non-negative integer $z$, 
\begin{equation}\label{BS-general}
\sum_{ \pi \in \mathcal{D}(n)  } (-1)^{ \# (\pi)-1 } \sum_{j=1}^{s(\pi)} 
( \ell(\pi) - s(\pi) + j ) ^z = \sigma_z(n),
\end{equation}
where $\sigma_{z}(n)=\sum_{d|n} d^z$. Later in 2022, Bhoria, Eyyunni and Maji \cite[Theorem 2.4]{BEM21} derived the following one-variable generalization of  \eqref{BS-general}. For any $z \in \mathbb{Z}$, $c \in \mathbb{C}$ and $n \in \mathbb{N}$, they proved that
\begin{equation}\label{One_Var_BS-general}
\sum_{ \pi \in \mathcal{D}(n)  } (-1)^{ \# (\pi)-1 } \sum_{j=1}^{s(\pi)} (\ell(\pi) - s(\pi)  + j ) ^z c^{\ell(\pi) - s(\pi) +j}=  \sigma_{z,c}(n),  
\end{equation}
where 
\begin{align*}
\sigma_{z,c}(n)=  \sum_{d | n} d^z c^d.
\end{align*}
Recently, the first author along with Bhoria, Eyyunni and Maji,  \cite[Theorem 2.1]{ABEM24} extended the domain of $z$ in \eqref{One_Var_BS-general} from integers to the set of complex numbers. In the same paper, they observed that the divisor function $d(n)$ can also be written as $p^{(1)}(n)$, where $p^{(1)}(n)$ counts the number of partitions of $n$ into exactly $1$ distinct part size, and rewrote \eqref{BS} as
\begin{align}\label{different form_BS}
p^{(1)}(n)=  \sum_{ \pi \in \mathcal{D}(n)  } (-1)^{ \# (\pi)-1 } s(\pi).  
\end{align}
Further, they \cite[Corollary 2.7]{ABEM24} gave an analogous result to \eqref{different form_BS}. We write their result as a corollary here as we intend to give a combinatorial proof of their result.
\begin{corollary}\label{main_corollary}
For any natural number $n$, one has
\begin{align}\label{Analogue of BS equation}
p^{(2)}(n) = \sum_{\pi \in \mathcal{D}(n)} (-1)^{ \# (\pi) }s(\pi) \left( \ell(\pi) - s(\pi) \right). 
\end{align}
\end{corollary}
Although the above result was established analytically, a Bressoud-Subbarao-type combinatorial proof of the identity \eqref{Analogue of BS equation} has not yet been given. In this paper, we fill this gap by providing a combinatorial proof of the identity \eqref{Analogue of BS equation}. We further generalize \eqref{Analogue of BS equation} for partitions into $k$ distinct part sizes. Mainly, we prove the following result.
\begin{theorem}\label{general_pk(n)}
 Let $p^{(k)}(n)$ be the number of partitions of $n$ into exactly $k$ distinct part sizes. Then, we have 
\begin{align}\label{general k identity}
p^{(k)}(n) = \sum_{\pi \in \mathcal{D}(n)} (-1)^{\#(\pi)-1} \binom{s(\pi)+k-1}{k} - \sum_{\substack{\lambda \in \mathcal{P}(n)\\ \nu(\lambda)\le k-1}} \binom{\ell(\lambda)+k-\nu(\lambda)-1} {k-\nu(\lambda)}.
\end{align}
\end{theorem}
\begin{remark}
For $k=1$, the sum over $\lambda$ is empty and Theorem \ref{general_pk(n)} reduces to \eqref{different form_BS}.
For $k=2$, we recover identity \eqref{Analogue of BS equation}.
\end{remark}

\section{Combinatorial Proofs}
We begin this section with a bijective proof of the following result, which serves as a key ingredient in establishing our main results.

\begin{lemma}\label{lemma 1}
For any $m,n \in \mathbb{N}$, we obtain
\begin{align}\label{lemma 1 equation}
\sum_{\substack{\pi\in\mathcal D(n)\\ s(\pi)=m}}(-1)^{ \# (\pi) -1}
=\sum_{\lambda\in\mathcal P(n)}(-1)^{\ell(\lambda)-m}\binom{\nu(\lambda)}{\ell(\lambda)-m}.
\end{align}
\end{lemma}
\begin{remark}
Identity \eqref{lemma 1 equation} is the coefficient of $c^{m}$, for $m \geq 1$, in the partition-theoretic form
\begin{align*}
\sum_{\pi\in\mathcal D(n)}(-1)^{\#(\pi)-1}\bigl(c^{s(\pi)}-1\bigr)
=\sum_{\lambda\in\mathcal P(n)}c^{\ell(\lambda)-\nu(\lambda)}(c-1)^{\nu(\lambda)}
\end{align*}
of an identity of Andrews, Garvan and Liang \cite[Theorem 3.5]{agl13}, which was recorded in this form by Dixit and Maji \cite[Equation~(2.6)]{DM}. The proofs in \cite{agl13,DM} use
$q$-series. To the best of our knowledge, Lemma \ref{lemma 1} gives the first
bijective proof of this refined identity.
\end{remark}
\begin{proof}
We start by simplifying RHS of \eqref{lemma 1 equation}. Let $\mu$ be the conjugate of $\lambda$. Then $\ell (\lambda) = \#(\mu)$ and $\nu (\lambda) = \nu (\mu)$. Thus, we have 
\begin{align*}
\sum_{\lambda\in\mathcal P(n)}(-1)^{\ell(\lambda)-m}\binom{\nu(\lambda)}{\ell(\lambda)-m} = \sum_{\mu\in\mathcal P(n)}(-1)^{\#(\mu)-m}\binom{\nu(\mu)}{\#(\mu)-m}. 
\end{align*}
Let $B_{\mu}$ be the subset of set of distinct parts of $\mu$ with $|B_{\mu}|=\#(\mu)-m$.
Then $\binom{\nu(\mu)}{\#(\mu)-m}$ counts the total number of such subsets. 
This gives 
\begin{align}\label{wieghted_sum_of_RHS}
\sum_{\mu\in\mathcal P(n)}(-1)^{\#(\mu)-m}\binom{\nu(\mu)}{\#(\mu)-m} =   \sum_{\mu\in\mathcal P(n)}(-1)^{|B_{\mu}|}\binom{\nu(\mu)}{|B_{\mu}|}= \sum_{(\mu, B_{\mu})} (-1)^{|B_{\mu}|},
\end{align}
where the sum runs over the ordered pair $(\mu, B_{\mu})$, which counts the number of subsets $B_{\mu}$ corresponding to each partition $\mu$ with weight $(-1)^{|B_{\mu}|}$ . For all $b \in B_{\mu}$, remove one copy of each $b$ from $\mu$ and call it $\kappa$. Then one can see that $
|\kappa|+\sum_{b \in B_{\mu}}b = n$ and $\#(\kappa) = m.$ It is easy to check that there is a one-one correspondence between pairs $(\mu, B_{\mu})$ and  $(\kappa, B)$, where $\kappa$ is any partition with exactly $m$ part sizes and $B$ is any finite set of positive integers with $|\kappa| + \sum_{b \in B}b= n.$ Adding one copy of each $b \in B$ to $\kappa$ recovers $\mu$.  Again we conjugate $\kappa$ and then remove one copy of $m$ and call it $\rho$. Then $\rho$ is a partition with all parts $\leq m$ and  $|\rho|+m+\sum_{b \in B_{\mu}}b=n$. Next step is to split the set $B_{\mu}$ as $B_{\mu}=S\sqcup T$ with $S=B_{\mu}\cap[1,m]$ and $T=B_{\mu}\cap(m,\infty)$. Thus, the weight in the last sum in \eqref{wieghted_sum_of_RHS} is $(-1)^{|S|}(-1)^{|T|}$.

Since all parts of $\rho$ are $\leq m$ allows us to fix $T$ and work with pair $(\rho, S)$. Let $i$ be the smallest integer in $[1,m]$ that lies either in $S$ or is a part of $\rho$. If none of the above happens, the pair is $(\emptyset,\emptyset)$, which is fixed. Thus, we want to show that for $(\rho, S) \neq (\emptyset,\emptyset)$, This pair must be paired with another pair. If $i\in S$, remove $i$ from $S$ and add a part $i$ to $\rho$. If $i\notin S$, so $i$ is a part of $\rho$, remove one part $i$ from $\rho$ and put $i$ into $S$. This preserves $|\rho|+\sum_{s \in S} s$, changes $|S|$ by one, and hence the associated weight is  $(-1)^{|S|-1}(-1)^{|T|}$.
Consequently only the terms with $|\rho|+\sum_{s \in S}s =0$, i.e.\ $\sum_{t \in T}t=n-m$, survive, each with weight $(-1)^{|T|}$. The sets $T\subset(m,\infty)$ with $\sum_{t \in T}t=n-m$ correspond exactly to $\pi=\{m\}\cup T\in\mathcal D(n)$ with $s(\pi)=m$ and $|T|= \#(\pi)-1$. This is the left side of \eqref{lemma 1 equation}.
\end{proof}
\begin{proof}[Theorem \ref{general_pk(n)}][]
Multiplying both sides of \eqref{lemma 1 equation} by $\binom{m+k-1}{k}$ and summing over all $m$, we obtain
\begin{align*}
\sum_{\pi \in \mathcal{D}(n)} (-1)^{ \# (\pi) -1} \binom{s(\pi)+k-1}{k} &= \sum_{m\geq 1} \sum_{\lambda\in \mathcal{P}(n)} (-1)^{\ell(\lambda)-m} \binom{\nu(\lambda)} {\ell(\lambda)-m} \binom{m+k-1}{k} \\
&= \sum_{\lambda\in P(n)} \sum_{m=1}^{\ell(\lambda)} (-1)^{\ell(\lambda)-m} \binom{\nu(\lambda)}
{\ell(\lambda)-m} \binom{m+k-1}{k}.
\end{align*}
Replacing $m$ by $\ell(\lambda)-m$, one has
\begin{align}\label{final_LHS}
\sum_{\pi \in \mathcal{D}(n)} (-1)^{ \# (\pi) -1} \binom{s(\pi)+k-1}{k} = \sum_{\lambda\in \mathcal{P}(n)} \sum_{m=0}^{\nu(\lambda)} (-1)^m \binom{\nu(\lambda)}{m} \binom{\ell(\lambda)-m+k-1}{k}.
\end{align}
The sum over $m$ runs till $\nu(\lambda)$ as $\binom{\nu(\lambda)}{m}$ vanishes for $m > \nu(\lambda)$. When $\nu(\lambda)=\ell(\lambda)$ the sum acquires the extra term $m=\ell(\lambda)$, which vanishes since $\binom{k-1}{k}=0$ for $k\ge 1$. Hence, the sum may be taken over $0\le m\le \nu(\lambda)$.
Let
\begin{align*}
f(m)=\binom{\ell(\lambda)-m+k-1}{k}.
\end{align*}
Then $f(m)$ is a polynomial of degree $k$ in $m$. Here we utilize the fact that if $f(x)$  is a polynomial of degree $d$, then
\begin{align*}
\sum_{j=0}^{n}(-1)^j
\binom{n}{j}f(j)=0,
\quad\text{if }n>d.
\end{align*}
Thus, \eqref{final_LHS} survives for $\nu(\lambda) \leq k$. This gives 
\begin{align}\label{general k first form}
\sum_{\pi \in \mathcal{D}(n)} (-1)^{ \# (\pi) -1} \binom{s(\pi)+k-1}{k} = \sum_{\substack{\lambda\in \mathcal{P}(n) \\ \nu(\lambda) \leq k}} \sum_{m=0}^{\nu(\lambda)} (-1)^m \binom{\nu(\lambda)}{m} \binom{\ell(\lambda)-m+k-1}{k}.
\end{align}
Now utilizing 
\begin{align*}
\sum_{m=0}^{v} (-1)^m \binom{v}{m} \binom{y-m}{k} = \binom{y-v}{k-v},
\end{align*}
with $v=\nu(\lambda)$ and $y=\ell(\lambda)+k-1$ in the right-hand side of \eqref{general k first form}, we obtain
\begin{align*}
\sum_{\pi \in \mathcal{D}(n)} (-1)^{ \# (\pi) -1} \binom{s(\pi)+k-1}{k} = \sum_{\substack{\lambda\in \mathcal{P}(n) \\ \nu(\lambda) \leq k}} \binom{\ell(\lambda)+k-\nu(\lambda)-1}{k-\nu(\lambda)}.
\end{align*}
Isolating the term corresponding to $\nu(\lambda)=k$ on the right-hand side of the above equation yields \eqref{general k identity}. This completes the proof of Theorem \ref{general_pk(n)}.
\end{proof}
\begin{proof}[Corollary \ref{main_corollary}][]
Letting  $k=2$ in \eqref{general k identity}, we obtain
\begin{align}\label{k=2_first_identity}
p^{(2)}(n) &=\sum_{\pi\in\mathcal{D}(n)} (-1)^{\#(\pi)-1}\binom{s(\pi)+1}{2} - \sum_{\substack{\lambda\in\mathcal{P}(n) \\ \nu(\lambda)=1}}\ell(\lambda) \nonumber \\
&=\sum_{\pi\in\mathcal{D}(n)} (-1)^{\#(\pi)-1}\binom{s(\pi)+1}{2} - \sigma(n).
\end{align}
Here, in the last step, we have utilized the fact that partitions with one distinct part size are $(d, d, …, d)$ with $d \mid n$, and their largest part is $d$. Utilize \eqref{BS-general} with $z=1$ to write $\sigma(n) = \sum_{\pi\in\mathcal{D}(n)} (-1)^{\#(\pi)-1}\sum_{j=1}^{s(\pi)} \bigl(\ell(\pi)-s(\pi)+j\bigr)$. Next, using the identity
\begin{equation*}
\sum_{j=1}^{s(\pi)} \bigl(\ell(\pi)-s(\pi)+j\bigr) - \binom{s(\pi)+1}{2} = s(\pi)\bigl(\ell(\pi)-s(\pi)\bigr),
\end{equation*}
on the RHS of \eqref{k=2_first_identity},  we obtain \eqref{Analogue of BS equation}. This completes the proof of Corollary \ref{main_corollary}
\end{proof}

\section{Concluding remarks}
We now explore several consequences of the identity \eqref{lemma 1 equation}. Multiplying both sides of \eqref{lemma 1 equation} by an arbitrary function $w(m)$ and summing over $m\geq 1$, we obtain the following identity, where we adopt the convention $w(0)=0$:
\begin{equation}\label{weight_general}
\sum_{\pi\in\mathcal D(n)}(-1)^{\#(\pi)-1}\,w(s(\pi))
=\sum_{\lambda\in\mathcal P(n)}\sum_{j=0}^{\nu(\lambda)}
(-1)^j\binom{\nu(\lambda)}{j}\,w(\ell(\lambda)-j).
\end{equation}
In particular, when $w(m)$ is a polynomial of degree $k$, the inner alternating sum vanishes for every partition $\lambda$ with $\nu(\lambda)>k$. Consequently, only partitions with at most $k$ distinct part sizes contribute to the right-hand side. Theorem~\ref{general_pk(n)} follows by taking $w(m)=\binom{m+k-1}{k}$. Other choices of $w(m)$ yield further identities. For instance, choosing $w(m)=m^k c^m$, where $k,c\in\mathbb C$, recovers \cite[Theorem 2.3]{ABEM24}. Similarly, taking $w(m)=\sum_{j=1}^{m}j^k c^j$, with $k,c\in\mathbb C$, yields \cite[Theorem 2.6]{ABEM24}. These examples illustrate the versatility of the identity \eqref{weight_general} and suggest that a wide range of partition identities can be derived by suitable choices of the weight function $w(m)$. It would be interesting to explore further choices of $w(m)$ and investigate the new combinatorial and arithmetic identities that may arise from this general framework.

\section*{Acknowledgements}
The authors wish to thank Dr. Pramod Eyyunni and Dr. Bibekananda Maji for giving useful suggestions and fruitful discussions.  The first author's research is funded by National Institute of Science Education and Research (NISER), Bhubaneswar. The second author is grateful to National Board of Higher Mathematics (NBHM) for providing Post Doctoral fellowship.  We sincerely thank NISER Bhubaneswar and IISc Bangalore for providing research-friendly environment. 
%\section{Data availability statement}
%The authors declare that the manuscript has no associated data. 
%\section{Statement of Conflict of Interest}
%On behalf of all authors,  the corresponding author states that there is no conflict of interest.  


\begin{thebibliography}{00}
 \bibitem{ABEM24} A.~Agarwal, S.~C.~Bhoria, P.~Eyyunni, and B.~Maji, \emph{Bressoud--Subbarao type weighted partition identities for a generalized divisor function}, Ann. Comb., \textbf{28} (2024), 555--574. 
 
 \bibitem{agl13} G.~E.~Andrews, F.~G.~Garvan,  and J. Liang, \emph{Self-conjugate vector partitions and the parity of the \textup{spt}-function}, Acta Arith., \textbf{158} (2013), 199--218.
 
 
 \bibitem{BEM21}
 S. ~C.~Bhoria,  P.~Eyyunni,  and B.~Maji,  \emph{Generalization of five $q$-series identities of Ramanujan and unexplored weighted partition identities},  Ramanujan J.,   {\bf 58} (2022),  435--462.  
 
 
 \bibitem{BS84}
 D.~Bressoud and M.~Subbarao, \emph{On Uchimura's connection between partitions and the number of divisors}, Canad.~Math.~Bull.,~\textbf{27} (1984), 143--145.
 
% \bibitem{overp}
%S.~Corteel and J.~Lovejoy, \emph{Overpartitions}, Trans. Amer. Math. Soc.,~\textbf{356} No. 4 (2004), 1623--1635.
 
%\bibitem{dilcher} 
%K.~Dilcher, \emph{Some $q$-series identities related to divisor functions}, Discrete Math., {\bf 145} (1995), 83--93.  
 
%  \bibitem{cdg}
%  S.H.~Chan, A.~Dixit and F.G.~Garvan, \emph{Rank-crank type PDEs and generalized Lambert series identities}, Ramanujan J.~\textbf{31}, No. 1-2 (2013), 163--189 (Special issue in honor of Mourad Ismail and Dennis Stanton).
 
%  \bibitem{chern}
%  S.~Chern, \emph{On a conjecture of George Beck}, Int.~J.~Number~Theory~\textbf{14} No. 3 (2018), 647--651. 

\bibitem{DM}
 A.~Dixit and B.~Maji, \emph{Partition implications of a three-parameter $q$-series identity}, Ramanujan J.,  {\bf 52} (2020), 323--358.

% \bibitem{DEMS} A.~Dixit, P.~Eyyunni, B.~Maji and G.~Sood, \emph{Untrodden pathways in the theory of the restricted partition function p(n,N)},  J.  Combin.  Theory Ser.  A {\bf 180} (2021), 105423.
 
% \bibitem{EMS} P.~Eyyunni, B.~Maji and G.~Sood, \emph{An Inequality between finite analogues of rank and crank moments},  Int.  J.  Number Theory,  {\bf 17},  405--423  (2021).({\bf Special issue in honor of Bruce C.  Berndt}).

%  \bibitem{dixityee}
%  A.~Dixit and A.~J.~Yee, \emph{Generalized higher order spt-functions}, Ramanujan J.~\textbf{31}, No. 1-2 (2013), 191--212 (Special issue in honor of Mourad Ismail and Dennis Stanton).
 
 %\bibitem{fine} N.~J.~Fine, \emph{Basic Hypergeometric Series and Applications}, Mathematical Surveys and Monographs, Amer. Math. Soc., 1989. 
 
 \bibitem{ffw95} R.~Fokkink, W.~Fokkink and Z.~B.~Wang, \emph{A relation between partitions and the number of divisors}, Amer. Math. Monthly, {\bf 102} (1995), 345--347.
 
 %\bibitem{fulascoux}
 %A.~M.~Fu and A.~Lascoux, \emph{$q$-Identities related to overpartitions and divisor functions}, Electron.~J.~Combin.~\textbf{12} (2005), \#R38.
 
 % \bibitem{garvan0}
%  F.~G.~Garvan, \emph{Weighted partition identities and divisor sums}, Talk, AMS special Session on `Partition Theory and Related Topics', Joint Mathematics Meeting, Atlanta, January 6, 2017.
 
% \bibitem{garvan1} F.~G.~Garvan, \emph{Weighted partition identities and divisor sums}, Ch. 12 in Frontiers in Orthogonal Polynomials and $q$-Series, Contemporary Mathematics and Its Applications: Monographs, Expositions and Lecture Notes: Vol. 1, M.~Z.~Nashed and X.~Li, eds., World Scientific, 2018, pp. 239--249.
 
%\bibitem{GasperRahman} G.~Gasper, M.~Rahman, Basic Hypergeometric series, Second Edition, (Encyclopedia of Mathematics and its applications), 2004. 
 
%\bibitem{glaisher} J.~W.~L.~Glaisher, \emph{On the representations of a number as the 
%sum of two, four, six, eight, ten, and twelve squares}, Quart. J. Math. ,  {\bf 38} (1907), 1--62.

%\bibitem{guozeng2015} V.~J.~W.~Guo and J.~Zeng, \emph{Basic and bibasic identities related to divisor functions}, 
%J. Math. Anal. Appl.,  {\bf 431} (2015), 1197--1209.
% 
% \bibitem{GK21} R.~Gupta and R.~Kumar,  \emph{On some $q$-series identities related to a generalized divisor function and their implications},  Discrete Math.,  {\bf 344} (2021),  112559.  
 
 
%\bibitem{Rgupta} R.~Gupta, \emph{On sum-of-tails identities}, arXiv:2002.00447,  2020. 
%  \bibitem{guozhang}
%   V.~J.~W.~Guo and C.~Zhang, \emph{Some further q-series identities related to divisor functions}, Ramanujan J.~\textbf{25} No. 3 (2011), 295--306. 
 
% \bibitem{hamme} L.~van Hamme, \emph{Problem $6407$}, Amer. Math. Monthly {\bf 89} No. 9 (1982), 703--704.
 
%  \bibitem{ismailstantonanncomb}
%  M.~E.~H.~Ismail and D.~Stanton, \emph{Some combinatorial and analytical identities}, Ann. Comb.~\textbf{16} No. 4 (2012), 755--771.
 
%  \bibitem{macmahon}
%  P.A.~Macmahon, \emph{Combinatory Analysis}, Vol.~II, Cambridge University Press, Cambridge, 1915-1916, reissued, Chelsea, 1960.
 
%  \bibitem{merca}
%  M.~Merca, \emph{A new look on the generating function for the number of divisors}, J.~Number Theory~\textbf{149} (2015), 57--69.
 
% \bibitem{kluyver} J.~C.~Kluyver, Vraagstuk XXXVII (Solution by S.C. van Veen), Wiskundige Opgaven (1919), 92--93.
 
 %\bibitem{patkowskijp}
 %A.~E.~Patkowski, \emph{A formula for the sum of three triangular numbers and some explorations}, JP J. Algebra, Number Theory Appl.~\textbf{10} No. 2 (2008), 185--191.
 
%  \bibitem{patkowskicm}
%  A.~E.~Patkowski, \emph{Divisors, partitions and some new $q$-series identities}, Colloq. Math.~\textbf{117} No. 2 (2009), 289--294.
 
 %\bibitem{prodingerrice}
 %H.~Prodinger, \emph{Some applications of the $q$-Rice formula}, Random Structures Algorithms~\textbf{19} (2001), 552--557.
 
% \bibitem{ramanujanoriginalnotebook2}
% S.~Ramanujan, \emph{Notebooks of Srinivasa Ramanujan, Vol. II}, Tata Institute of Fundamental Research, Bombay, 1957, 393 pp.
 
% \bibitem{lnb}
% S.~Ramanujan, \emph{The Lost Notebook and Other Unpublished  Papers}, Narosa, New Delhi, 1988.
 
% \bibitem{ramanujantifr}
%  S.~Ramanujan, \emph{Notebooks of Srinivasa Ramanujan, Vol. II}, Tata Institute of Fundamental Research, Mumbai, 2012.

% \bibitem{sloane}
% N.~J.~A.~Sloane, http://oeis.org/A228441.
 
\bibitem{uchimura81} K.~Uchimura, \emph{An identity for the divisor generating function arising from sorting theory}, J. Combin.
 Theory Ser. A, {\bf 31} (1981), 131--135. 
 
%\bibitem{uchimura87} K.~Uchimura,  \emph{Divisor generating functions and insertion into a heap},  Discrete Appl. Math.,  {\bf 18} (1987), 73--81.   
 
%  \bibitem{wangyee}
%  L.~Wang and A.~J.~Yee, \emph{Some Hecke-Rogers type identities}, preprint.
 
 
 \end{thebibliography}
\end{document}